\documentclass[11pt,a4paper,twoside,bibliography=totoc]{scrartcl}
\usepackage{a4wide}
\usepackage{amsfonts}
\usepackage{amsmath}
\usepackage{amssymb}
\usepackage{array}
\usepackage{amsthm}
\usepackage{subfigure}
\usepackage{subfig}
\usepackage{mathdots}

\usepackage{algorithm}
\usepackage{algorithmic}

\usepackage{graphicx}
\usepackage{color}
\usepackage{pgf}
\usepackage{scrpage2}
\usepackage{multirow}
\usepackage{listings} 
\usepackage{tikz}
\usepackage{graphicx}
		
\newcommand{\N}{\ensuremath{\mathbb{N}}}
\newcommand{\T}{\ensuremath{\mathbb{T}}}
\renewcommand{\S}{\ensuremath{\mathbb{S}}}

\newcommand{\NZ}{\ensuremath{\mathbb{N}_{0}}}
\newcommand{\Z}{\ensuremath{\mathbb{Z}}}

\newcommand{\R}{\ensuremath{\mathbb{R}}}
\newcommand{\C}{\ensuremath{\mathbb{C}}}
\newcommand{\K}{\ensuremath{\mathbb{K}}}

\newcommand{\ii}{\textnormal{i}}

\newcommand{\eip}[1]{\textnormal{e}^{2\pi\ii{#1}}}

\newcommand{\set}[1]{\left\{ #1 \right\}}
\newcommand{\abs}[1]{\left| #1 \right |}

\DeclareMathOperator*{\diag}{diag}

\DeclareMathOperator*{\maxdeg}{maxdeg}
\DeclareMathOperator*{\sep}{sep}
\DeclareMathOperator*{\lt}{lt}
\DeclareMathOperator{\V}{\mathcal{V}}
\DeclareMathOperator{\I}{\mathcal{I}}

\newtheorem{thm}{Theorem}[section]
\newtheorem{lemma}[thm]{Lemma}
\newtheorem{remark}[thm]{Remark}
\newtheorem{definition}[thm]{Definition}
\newtheorem{example}[thm]{Example}
\newtheorem{corollary}[thm]{Corollary}
\newtheorem{proposition}[thm]{Proposition}

\numberwithin{equation}{section}
\numberwithin{table}{section}
\numberwithin{figure}{section}

\newcommand{\bend}{\hspace*{0ex} \hfill \hbox{\vrule height
    1.5ex\vbox{\hrule width 1.4ex \vskip 1.4ex\hrule  width 1.4ex}\vrule
    height 1.5ex}}

\long\def\symbolfootnote[#1]#2{\begingroup%
\def\thefootnote{\fnsymbol{footnote}}\footnote[#1]{#2}\endgroup}

\newcommand{\calA}{\mathcal{A}} 
\newcommand{\MM}{\mathcal{M}}

\newcommand{\dd}{\mathrm{d}}

\renewcommand{\mathbf}[1]{\ensuremath{\boldsymbol{#1}}}

\newcommand{\rank}{ \operatorname{rank}}

\renewcommand{\thefootnote}{\fnsymbol{footnote}}

\allowdisplaybreaks
\title{Prony's method on the sphere}

\date{\today}
\date{}

\author{Stefan Kunis\footnotemark[2]\space\footnotemark[3]
\and H.~Michael M\"oller\footnotemark[4]
\and Ulrich von der Ohe\footnotemark[2]}

\begin{document}

\maketitle

\begin{abstract}

\noindent
Eigenvalue analysis based methods are well suited for the reconstruction of finitely supported measures from their moments up to a certain degree.
We give a precise description when Prony's method succeeds in terms of an interpolation condition.
In particular, this allows for the unique reconstruction of a measure from its trigonometric moments whenever its support is separated and
also for the reconstruction of a measure on the unit sphere from its moments with respect to spherical harmonics.
Both results hold in arbitrary dimensions and also yield a certificate for popular semidefinite relaxations of these reconstruction problems.

\medskip

\medskip

\end{abstract}

\footnotetext[2]{
  Osnabr\"uck University, Institute of Mathematics
  \texttt{\{skunis,uvonderohe\}@uos.de}
}

\footnotetext[3]{
  Helmholtz Zentrum M\"unchen, Institute of Computational Biology
}

\footnotetext[4]{
  TU Dortmund, Fakult\"at f\"ur Mathematik
  \texttt{moeller@mathematik.tu-dortmund.de}
}

\section{Introduction}
\label{sect:intro}

Prony's method \cite{Pr95}, see also e.g.~\cite{PoTa12b,PlTa13}, reconstructs the coefficients and distinct parameters $\hat{f}_j,x_j\in\C$, $j=1,\hdots,M$, of the Dirac ensemble
$\mu=\sum_{j=1}^M\hat{f}_j\delta_{x_j}$
from the $2M+1$ moments
\begin{equation*}
  f(k)=\int_{\C}x^k\dd\mu(x)
  =\sum_{j=1}^M\hat{f}_jx_j^k,
  \quad
  k=0,\hdots,2M.
\end{equation*}
The computation of the parameters $x_j$ is done by setting up a certain Hankel or Toeplitz matrix of these moments and computing the roots of the polynomial with the monomial coefficients given by any non-zero kernel vector of this matrix.
Afterwards, the coefficients $\hat{f}_j$ can be computed by solving a Vandermonde linear system of equations.

We recently generalized this prototypical algorithm to the multivariate case by realizing the parameters as common roots of $d$-variate polynomials belonging to the
kernel of a certain multilevel Hankel or Toeplitz matrix \cite{KuPeRoOh16}.
In the present paper, we give a precise description when the parameters can be identified in terms of a simple interpolation condition, which in turn is
equivalent to the surjectivity of a certain evaluation homomorphism and also to the full rank of a certain Vandermonde matrix.
Since identifiability also implies full rank of a slightly larger Vandermonde matrix, we end up with a variant of the so called flat extension property \cite{CuFi00,LaMo09}.
Beyond this, our characterization with respect to the Vandermonde matrix allows to derive simple geometric conditions on the parameters:
given the order of the moments is bounded from below by some explicit constant divided by the separation distance of the parameters, unique reconstruction is guaranteed.
In particular, we get rid of the commonly stated technical condition that the order has to be larger than the number of parameters and weaken the
`coordinate wise' separation condition as used in \cite{Li15} to a truly multivariate separation condition.
Moreover, studying the Vandermonde-like factorization of the Hankel-like matrix of moments allows for a transparent generalization to Dirac ensembles on the sphere where
only moments with respect to the spherical harmonics are used for reconstruction.
Recently, the considered problem has also been studied as a constrained total variation minimization problem on the space of measures and attracted quite some attention, see
e.g.~\cite{BrPi13,DuPe15} and references therein.
As a corollary to our results, we give a painless construction of a so-called dual certificate for the total variation minimization problem on the $d$-dimensional torus \cite{CaFe13,CaFe14} and on the $d$-dimensional sphere \cite{BeDeFe15a,BeDeFe15b}.
Since our construction is a sum of squares, this construction also bypasses the relaxation from a nonnegative polynomial to the semidefinite program which is in general known
to possibly increase degrees for $d=2$ and to possibly fail for $d>2$, cf.~\cite[Remark~4.17, Theorem~4.24, and Remark~4.26]{Du07}.
We finally close by two small numerical examples, postpone a detailed study of stability and computational times to a future exposition, and give a short summary.

\section{Preliminaries}\label{sect:pre}

Throughout the paper, $\K$ denotes a field and $d\in\N$ denotes a natural number.
For $x\in\K^d$, $k\in\NZ^d$, we use the multi-index notation $x^k:=x_1^{k_1}\dotsm x_d^{k_d}$.
We start by defining the object of our interest, that is, multivariate exponential sums, as a natural generalization of univariate exponential sums.
\begin{definition}\label{def:mvexpsum}
A function $f\colon\NZ^d\to\K$ is a \emph{$d$-variate exponential sum}
if there are $M\in\N$, $\hat{f}_1,\hdots,\hat{f}_M\in\K$, and pairwise distinct $x_1,\hdots,x_M\in\K^d$
such that we have
\begin{equation*}
  f(k)  =\sum_{j=1}^M\hat{f}_jx_j^k
\end{equation*}
for all $k\in\NZ^d$.
In that case $M$, $\hat{f}_j$, and $x_j$, $j=1,\hdots,M$, are uniquely determined, and $f$ is called~\emph{$M$-sparse},
the $\hat{f}_j$ are called \emph{coefficients} of $f$,
and $x_j$ are called \emph{parameters} of $f$.
The set of parameters of $f$ is denoted by $\Omega=\Omega_f:=\{x_j:j=1,\hdots,M\}$.
\end{definition}
Let $f\colon\NZ^d\to\K$ be an $M$-sparse $d$-variate exponential sum with coefficients $\hat{f}_j\in\K$ and parameters $x_j\in\K^d$, $j=1,\dots,M$.
Our objective is to reconstruct the coefficients and parameters of $f$ given a finite set of samples of $f$ at a subset of $\NZ^d$, see also~\cite{PePlSc15}.

The following notations will be used throughout the paper.
For $k,n\in\NZ^d$ let $\abs{k}=\sum_{j=1}^dk_j$ and $N:=\binom{n+d}{d}$.
The matrix
\begin{equation*}
  H_n:=H_n(f):=\left(f(k+\ell)\right)_{k,\ell\in \NZ^d,|k|,|\ell|\le n}\in\K^{N\times N}
\end{equation*}
will play a crucial role in the multivariate Prony method.
Note that its entries are sampling values of $f$ at a grid of $\binom{2n+d}{d}$ integer points and that it is a sub-matrix of the multilevel Hankel matrix
$(f(k+\ell))_{k,\ell\in\{0,\dots,n\}^d}$.

Next we establish the crucial link between the matrix $H_n$ and the roots of multivariate polynomials.
To this end, let $\Pi:=\K[X_1,\dots,X_d]$ denote the $\K$-algebra of $d$-variate polynomials over $\K$
and for $p=\sum_kp_kX_1^{k_1}\dotsm X_d^{k_d}\in\Pi\setminus\{0\}$ let
$\deg(p):=\max\{\abs{k}:p_k\ne0\}$.
The $N$-dimensional sub-vector space of $d$-variate polynomials of degree at most $n$ is
\begin{equation*}
  \Pi_n:=\{p\in\Pi\setminus\{0\}:\deg(p)\le n\}\cup\{0\}.
\end{equation*}
For arbitrary $V\subset\K^d$, the \emph{evaluation homomorphism} at $V$ will be denoted by
\begin{equation*}
  \calA^V\colon\Pi\to\K^V,  \quad  p\mapsto(p(x))_{x\in V},
\end{equation*}
and its restriction to the sub-vector space $\Pi_n\subset\Pi$ will be denoted by $\calA_n^V$.
Note that the representation matrix of $\calA_n=\calA_n^\Omega$ with $\Omega=\{x_1,\dots,x_M\}$
w.r.t.~the canonical basis of $\K^M$ and the monomial basis of $\Pi_n$
is given by the multivariate Vandermonde matrix
\begin{equation*}
  A_n=\big(x_j^k\big)_{\substack{j=1,\dots,M\\k\in\NZ^d,|k|\le n}}\in\K^{M\times N}.
\end{equation*}
The connection between the matrix $H_n$ and polynomials that vanish on $\Omega$ lies in the observation that, using Definition~\ref{def:mvexpsum},
the matrix $H_n$ admits the factorization
\begin{equation}\label{eq:factorH}
  H_n=(f(k+\ell))_{k,\ell\in \NZ^d,|k|,|\ell|\le n}=A_n^{\top}DA_n,
\end{equation}
with $D=\diag(\hat{f}_1,\dots,\hat{f}_M)$.
Therefore the kernel of $A_n$, corresponding to the polynomials in $\Pi_n$ that vanish on $\Omega$, is a subset of the kernel of $H_n$.

In order to deal with the multivariate polynomials encountered in this way we need some additional notation.
The \emph{zero locus} of a set $P\subset\Pi$ of polynomials is denoted by
\begin{equation*}
\V(P):=\{x\in\K^d:\text{$p(x)=0$ for all $p\in P$}\},
\end{equation*}
that is,
$\V(P)$ consists of the \emph{common} roots of all the polynomials in $P$.
For a set $V\subset\K^d$, the kernel of $\calA^V$ (which is an ideal of $\Pi$) will be denoted $\I(V)$
and is called the \emph{vanishing ideal} of $V$; it consists of all polynomials that vanish on $V$.
Further, let $\I_n(V):=\ker\calA_n^V=\I(V)\cap\Pi_n$ denote the $\K$-sub-vector space of polynomials of degree at most~$n$ that vanish on $V$.
Subsequently, we identify $\Pi_n$ and $\K^N$
and switch back and forth between
matrix-vector and polynomial notation.
In particular,
we do not necessarily distinguish between
$\calA_n^\Omega$ and its representation matrix $A_n$,
so that e.g.~``$\V(\ker A_n)$'' makes sense.

\section{Main results}\label{sect:main}

We proceed with a general discussion that the identifiability of the parameters and an interpolation at these parameters are almost equivalent.
While this is closely related to the so-called flat extension principle \cite{CuFi00,LaMo09}, we also give a refinement which is of great use when
discussing the moment problem on the sphere.
The second and third subsection study the trigonometric moment problem and the moment problem on the unit sphere, respectively.
In both cases, appropriate separation conditions guarantee the above mentioned interpolation condition and thus identifiability of the parameters.
As a corollary, we give a simple construction of a dual certificate for the total variation minimization problem on the $d$-dimensional torus \cite{CaFe13,CaFe14} and
on the $d$-dimensional sphere \cite{BeDeFe15a,BeDeFe15b}.

\subsection{Interpolation and vanishing ideals}\label{subsect:suff}
We recall some notions from the theory of Gr\"obner bases which are needed in this section, see e.g.~\cite{BeWe93,KrRo00,CoLiSh07}.
A \emph{$d$-variate term} is a polynomial
of the form $X^k=X_1^{k_1}\dotsm X_d^{k_d}$ for some $k=(k_1,\dots,k_d)\in\NZ^d$.
The monoid of all $d$-variate terms will be denoted
$\MM:=\set{X^k:k\in\NZ^d}$.
A \emph{term order} on $\MM$
is a linear order~$\le$ on $\MM$
such that $1\le t$ for all $t\in\MM$
and $t_1\le t_2$ implies $t_1t_3\le t_2t_3$ for all $t_1,t_2,t_3\in\MM$.
For a polynomial $p=\sum_kp_kX^k\in\Pi\setminus\set{0}$
let $\lt_\le(p):=\max_\le\set{X^k:p_k\ne0}$
and for an ideal $I\ne\set{0}$ of $\Pi$
let $\lt_\le(I):=\set{\lt_\le(p):p\in I\setminus\set{0}}$.
The set $\mathcal{N}_\le(I):=\MM\setminus\lt_\le(I)$
is called \emph{normal set} of $I$.
A term order $\le$ is \emph{degree~compatible}
if $t_1\le t_2$ implies $\deg(t_1)\le\deg(t_2)$,
or equivalently,
if $\deg(p)=\deg(\lt_\le(p))$ for all $p\in\Pi\setminus\set{0}$.

\begin{lemma}[see e.g.~\protect{\cite[Prop.~2.6]{FaMo15}}]\label{lem:FassinoMoeller}
Let~$\le$ be a
term order on $\MM$.
If $I$ is an ideal of $\Pi$
and $t\in\mathcal{N}_\le(I)$,
then $t$ is the least element of
$M_t:=\set{\lt(p):\text{$p\in\Pi\setminus\set{0}$, $\overline{p}=\overline{t}$ in $\Pi/I$}}$.
\end{lemma}
\begin{proof}
Since $\lt(t)=t$,
we have $t\in M_t$.
Let $p\in\Pi\setminus\set{0}$ with $\overline{p}=\overline{t}$ in $\Pi/I$.
We have to show $t\le\lt(p)$.
Without loss of generality we can assume $p\ne t$.
Thus let $t-p=\sum_kc_kX^k\ne0$ with $c_k\in\K$ and $X^m=\lt(t-p)$.

\emph{Case~1:}~For every $k$, $t\ne c_kX^k$.
Then, since $p=t-\sum_kc_kX^k$,
we have $t\le\lt(p)$.

\emph{Case~2:}~There is a $k$ such that $t=c_kX^k$.
Then $c_k=1$ and we have $t=X^k\le X^m$
and since $t\in\mathcal{N}(I)=\MM\setminus\lt(I)$
and $X^m=\lt(t-p)\in\lt(I)$,
we have $t<X^m$.
Therefore we have $t<X^m=\lt(t-\sum_kc_kX^k)=\lt(p)$.
\end{proof}

\begin{lemma}\label{lem:Moeller}
Let~$\le$ be a degree~compatible term order on $\MM$.
Let $\Omega\subset\K^d$ be finite
and
$n\in\NZ$ such that the evaluation homomorphism
$\calA_n^\Omega\colon\Pi_n\to\K^\Omega$
is surjective.
Then $\mathcal{N}_\le(\I(\Omega))\subset\Pi_n$.
\end{lemma}
\begin{proof}
Since $\Omega$ is finite,
$\I(\Omega)\ne\set{0}$.
Let $t\in\mathcal{N}(\I(\Omega))$
and consider $\overline{t}$ in $\Pi/\I(\Omega)$.
Since $\I(\Omega)=\bigcap_{a\in\Omega}\I(a)$
and $\I(a)$, $a\in\Omega$, are pairwise co-prime,
by the Chinese remainder theorem
$\varphi\colon\Pi/\I(\Omega)\to\K^\Omega$,
$\overline{p}\mapsto(p(a))_{a\in\Omega}$,
is a bijection.
Since $\calA_n^\Omega\colon\Pi_n\to\K^\Omega$ is surjective,
there is a $p\in\Pi_n$ with
$\varphi(\overline{t})
=\calA_n^\Omega(p)=(p(a))_{a\in\Omega}
=\varphi(\overline{p})$,
hence $\overline{p}=\overline{t}$.
Since $t\notin\I(\Omega)$,
in particular $p\ne0$.
Thus Lemma~\ref{lem:FassinoMoeller}
together with the degree~compatibility of~$\le$
implies
\begin{equation*}
\deg(t)
=\deg(\min\set{\lt q:\text{$q\ne0$, $\overline{q}=\overline{t}$ in $\Pi/\I(\Omega)$}})
=\min\set{\deg(\lt q):\overline{q}=\overline{t}}
\le\deg(p)
\le n,
\end{equation*}
i.e.~$t\in\Pi_n$.
\end{proof}

\begin{thm}\label{thm:Moeller}
Let $\emptyset\ne\Omega\subset\K^d$ be finite
and $n\in\NZ$ such that
$\calA_n^\Omega\colon\Pi_n\to\K^\Omega$
is surjective.
Then $\Omega=\V(\I_{n+1}(\Omega))$.
\end{thm}
\begin{proof}
Let~$\le$ be a degree~compatible term order on $\MM$
and let
\begin{equation*}
P:=\set{t\in\lt(\I(\Omega)):\text{$t$ $\mid$-minimal in $\lt(\I(\Omega))$}}.
\end{equation*}
We show that $P\subset\Pi_{n+1}$.
Let $t\in P$.
Since $\Omega\ne\emptyset$,
$t\ne1$.
Thus $t=X_jt^\prime$ for some $j\in\set{1,\dots,d}$ and $t^\prime\in\MM$.
By $\mid$-minimality of~$t$ in $\lt(\I(\Omega))$,
we have $t^\prime\notin\lt(\I(\Omega))$.
Thus $t^\prime\in\mathcal{N}(\I(\Omega))$
which implies $\deg(t^\prime)\le n$ by Lemma~\ref{lem:Moeller}
and hence $\deg(t)=\deg(X_jt^\prime)=\deg(X_j)+\deg(t^\prime)\le n+1$,
i.e.~$t\in\Pi_{n+1}$.

By Dickson's lemma
(cf.~\cite[Thm.~5.2 and Cor.~4.43]{BeWe93}),
$P$ is finite.
Thus let $P=\set{t_1,\dots,t_r}$ with pairwise different $t_j$
and $g_1,\dots,g_r\in\I(\Omega)$ with $\lt g_j=t_j$.
Then $G=\set{g_1,\dots,g_r}$ is a (Gr\"obner) basis for $\I(\Omega)$
(see e.g.~Becker-Weispfenning~\cite[Prop.~5.38~(iv)]{BeWe93})
and $\deg(g_j)=\deg(\lt g_j)=\deg(t_j)\le n+1$,
i.e.~$G\subset\I_{n+1}(\Omega)$.
In particular
$\Omega
\subset\V(\I_{n+1}(\Omega))
\subset\V(G)
=\V(\langle G\rangle)
=\V(\I(\Omega))=\Omega$,
since $\Omega$ is finite (as usual, $\langle G\rangle$ denotes the ideal generated by $G$).
\end{proof}

\begin{remark}\label{rem:identifiability-characterization}
In summary, for every subset $\Omega\subset\K^d$ with $\lvert\Omega\rvert=M\in\N$ we have the chain of implications
 \begin{equation*}
 \rank \calA_n^\Omega=M \quad  \Rightarrow \quad \Omega=\V(\I_{n+1}(\Omega)) \quad  \Rightarrow \quad \rank \calA_{n+1}^\Omega=M,
\end{equation*}
where the second implication follows as in the proof of \cite[Thm.~3.1]{KuPeRoOh16}.
Moreover note that the factorization \eqref{eq:factorH} and Frobenius' rank inequality \cite[0.4.5~(e)]{HoJo13} implies
\begin{equation*}
2\rank \calA_n^\Omega\le\rank H_n+M\le\rank \calA_n^\Omega+M,\quad n\in\N, 
\end{equation*}
and thus the equivalence
\begin{equation*}
 \rank H_n=\rank H_{n+1}=M \quad \Leftrightarrow \quad \rank \calA_n^\Omega=M,
\end{equation*}
where the left hand side is exactly the flat extension principle \cite{CuFi00,LaMo09}.
We would like to note that considering $\calA_n$ allows for signed measures and yields simple \emph{a-priori} conditions on the order of the moments, see Lemmata \ref{lem:ingham} and \ref{lem:inghamSd}, while the flat extension principle is an \emph{a-posteriori} test and can in particular be used to find the possibly unknown number of parameter~$M$.
\end{remark}

In order to give a slight refinement of Theorem~\ref{thm:Moeller}
in Corollary~\ref{cor:Moeller}
we need the following notation.
For a set $V\subset\K^d$ let $\Pi_V:=\Pi/\I(V)$
and $\Pi_{V,n}:=\Pi_n/\I_n(V)$.
The map $\Pi_V\to\set{p_{\vert V}:p\in\Pi}$,
$p+\I(V)\mapsto p_{\vert V}$,
(where we use the same notation for a polynomial~$p\in\Pi$
and its induced polynomial function~$p\colon\K^d\to\K$)
is a ring isomorphism.
Thus we may identify the residue class $\overline{p}=p+\I(V)$
of $p\in\Pi$ with the function $p_{\vert V}\colon V\to\K$.
Since the $\K$-vector space homomorphism $\Pi_n\to\Pi_V$, $p\mapsto\overline{p}$,
has $\I_n(V)$ as its kernel,
$\Pi_{V,n}$ is embedded in $\Pi_V$.
The $\K$-vector space $\Pi_{V,n}$ is isomorphic to $\set{p_{\vert V}:p\in\Pi_n}$
by mapping $p+\I_n(V)$ with $p\in\Pi_n$ to $p_{\vert V}$.
For $\Omega\subset V$ let
$\calA_V^\Omega\colon\Pi_V\to\K^\Omega$,
$\overline{p}\mapsto\calA^\Omega(p)$,
which is well-defined by the above,
and let $\calA_{V,n}^\Omega$ denote the restriction of $\calA_V^\Omega$
to the $\K$-sub-vector space $\Pi_{V,n}$ of $\Pi_V$.
Further let $\I_{V,n}(\Omega):=\ker\calA_{V,n}^\Omega$.
For a set $Q\subset\Pi_V$ let
$\V_V(Q):=\set{a\in V:\text{$\overline{q}(a)=0$ for all $\overline{q}\in Q$}}$.

\begin{lemma}\label{lem:V_V(I_(V,n)(Omega))}
Let $V\subset\K^d$,
$\Omega\subset V$
and $n\in\NZ$.
Then we have
\begin{equation*}
\Omega\subset\V_V(\I_{V,n}(\Omega))\subset\V(\I_n(\Omega)).
\end{equation*}
\end{lemma}

\begin{proof}
The first inclusion is clear.
To prove the second inclusion,
let $a\in\V_V(\I_{V,n}(\Omega))$
and $p\in\I_n(\Omega)=\ker\calA_n^\Omega$.
We have to show that $p(a)=0$.
Let $\overline{p}:=p+\I_n(V)$.
Since $p\in\Pi_n$,
$\overline{p}\in\Pi_n/\I_n(V)=\Pi_{V,n}$
and we have
$\calA_{V,n}^\Omega(\overline{p})
=\calA_V^\Omega(\overline{p})
=\calA^\Omega(p)=0$,
i.e.~$\overline{p}\in\ker\calA_{V,n}^\Omega=\I_{V,n}(\Omega)$.
Since $a\in\V_V(\I_{V,n}(\Omega))$,
that is,
$a\in\V$ and $\overline{q}(a)=0$ for all $\overline{q}\in\I_{V,n}(\Omega)$,
it follows that $p(a)=\overline{p}(a)=0$.
\end{proof}

Combining this with Theorem~\ref{thm:Moeller}
yields the following.

\begin{corollary}\label{cor:Moeller}
Let $V\subset\K^d$,
$\Omega$ be a non-empty finite subset of $V$
and $n\in\NZ$ such that $\calA_{V,n}^\Omega$ is surjective.
Then
\begin{equation*}
\Omega=\V_V(\I_{V,n+1}(\Omega)).
\end{equation*}
\end{corollary}

\begin{proof}
Since $\calA_{V,n}^\Omega\colon\Pi_{V,n}\to\K^\Omega$ is surjective,
$\calA_n^\Omega\colon\Pi_n\to\K^\Omega$ is clearly also surjective.
Therefore we can apply Theorem~\ref{thm:Moeller}
which together with Lemma~\ref{lem:V_V(I_(V,n)(Omega))}
yields
\begin{equation*}
\Omega\subset\V_V(\I_{V,n+1}(\Omega))\subset\V(\I_{n+1}(\Omega))=\Omega.
\end{equation*}
\end{proof}

\subsection{Trigonometric polynomials and parameter on the torus}\label{subsect:trigonometric}

Now let $\K=\C$ and restrict to parameters on the $d$-dimensional torus $\T^d:=\{z\in\C:|z|=1\}^d$ with parameterization
$\T^d\ni z=\eip{t}$ for a unique $t\in[0,1)^d$.
Now, let $M\in\N$, coefficients $\hat{f}_j\in\C\setminus\set{0}$, and pairwise distinct $t_j\in[0,1)^d$, $j=1,\dots,M$,
be given.
Then the trigonometric moment sequence of the complex Dirac ensemble
$\tau:\mathcal{P}([0,1)^d)\rightarrow\C$, $\tau=\sum_{j=1}^M\hat{f}_j\delta_{t_j}$, is the
the $d$-variate exponential sum
\begin{equation*}
f\colon\Z^d\to\C,
\quad
k\mapsto\int_{[0,1)^d}\eip{kt}\dd\tau(t)
=\sum_{j=1}^M\hat{f}_j\eip{kt_j},
\end{equation*}
with parameters $\eip{t_j}=(\eip{t_{j,1}},\hdots,\eip{t_{j,d}})\in\T^d$.

A convenient choice for the truncation of this sequence is $|k|_{\infty}=\max\{|k_1|,\hdots,|k_d|\}\le n$.
We define the multivariate Vandermonde matrix a.k.a. nonequispaced Fourier matrix
\begin{equation*}
 F_n:=\left(\eip{k t_j}\right)_{\substack{j=1,\dots,M\\k\in\N_0^d,|k|_{\infty}\le n}}\in\C^{M\times(n+1)^d}.
\end{equation*}

\begin{lemma}\label{lem:Moeller-maxdeg}
Let $\Omega:=\{\eip{t_j}:t_j\in [0,1)^d,\;j=1,\hdots,M\}\subset\T^d$ and $n\in\N_0$ such that $F_n$ has full rank $M$, then
$\Omega=\V(\ker F_{dn+1})$.
\end{lemma}
\begin{proof}
First note that $\{k\in\N_0^d:|k|_{\infty}\le n\}=\{0,\hdots,n\}^d\subset I_{dn}$ and thus
${\calA}_{dn}^\Omega$ is surjective and Theorem~\ref{thm:Moeller} yields $\Omega=\V(\ker{\calA}_{dn+1}^\Omega)$.
Finally note that $I_{dn}\subset \{0,\hdots,dn+1\}^d$ and thus the result follows from $\V(\ker{\calA}_{dn+1}^\Omega)\supset\V(\ker F_{dn+1})$.
\end{proof}

\begin{remark}
It is tempting to try to prove Lemma~\ref{lem:Moeller-maxdeg} with~$n+1$
instead of~$dn+1$ analogously to Theorem~\ref{thm:Moeller} by using
``$\maxdeg$-compatible'' term orders instead of degree compatible term orders.
However, for $d\ge2$ there are no such term orders. To see this, let~$\le$ be
a term order on $\MM$ and w.l.o.g.~let $X_2\le X_1$.
Then $X_2^2\le X_1X_2$ and $\maxdeg(X_2^2)=2>1=\maxdeg(X_1X_2)$.
\end{remark}

\begin{lemma}[\protect{\cite[Lem.~3.1]{PoTa132}}]\label{lem:ingham}
 For $\Omega:=\{\eip{t_j}:t_j\in [0,1)^d,\;j=1,\hdots,M\}\subset\T^d$ let 
  \begin{equation*}
   \sep(\Omega):=\min_{r\in\Z^d,\;j\ne\ell}\| t_j-t_{\ell}+r\|_{\infty}
  \end{equation*}
 denote the separation distance and call the set of parameters $q$-separated if $\sep(\Omega)>q$.
 Now if $n\in\NZ$ fulfills $n>\sqrt{d}/q$, then the matrix $F_n\in\C^{M\times (n+1)^d}$ has full rank $M$.
\end{lemma}

\begin{remark}
 The semi-discrete Ingham inequality \cite[Ch.~8]{KoLo04} has been made fully discrete in \cite[Lem.~3.1]{PoTa132}.
 Equivalent results are given in \cite{Ba99,KuPo07} as condition number estimates for Vandermonde matrices.
 More recently, a sharp condition number estimate for the univariate case $d=1$ has been proven in \cite{Mo15} and a
 multivariate generalization under `coordinate wise separation' has been given in \cite{Li15}.
\end{remark}

\begin{thm}\label{thm:trigMain}
 Let $f\colon\Z^d\to\C$ be an $M$-sparse $d$-variate exponential sum with parameters $x_j\in\T^d$, $j=1,\dots,M$.
 If the parameters are $q$-separated and $n>d^{3/2}/q+d+1$, then
\begin{equation*}
 \Omega=\V(\ker T_n),
\end{equation*}
where the entries of the matrix are given by trigonometric moments of order up to $n$, i.e.,
\begin{equation*}
 T_n=(f(k-\ell))_{k,\ell\in\{0,\hdots,n\}^d}\in \C^{(n+1)^d\times (n+1)^d}.
\end{equation*}
\end{thm}
\begin{proof}
Setting $n_0:=\lfloor (n-1)/d\rfloor$ yields $n_0>\sqrt{d}/q$ and Lemma \ref{lem:ingham} implies full rank of $F_{n_0}$.
Thus $\Omega=\V(\ker F_n)$ is guaranteed by Lemma \ref{lem:Moeller-maxdeg} and the factorization
\begin{equation*}
T_n=F_n^* D F_n,\quad D=\diag(\hat{f}_1,\dots,\hat{f}_M),
\end{equation*}
being a variant of \eqref{eq:factorH}, together with the Frobenius' rank inequality~\cite[0.4.5~(e)]{HoJo13}
\begin{equation*}
M
=\rank F_n^* D+\rank D F_n-\rank D
\le\rank T_n
\le\rank F_n
=M,
\end{equation*}
implies $\ker F_n=\ker T_n$ from which the assertion follows.
\end{proof}

This improves over \cite[Thm.~3.1, 3.7]{KuPeRoOh16} by getting rid of the technical condition $n\ge M$ and thus
the number of used moments can be bounded from above by $(n+1)^d\le C_d M$ if the parameters are quasi-uniformly distributed.
Finally, note that the sum of squares representation \cite[Thm.~3.5]{KuPeRoOh16} implies that $n>d^{3/2}/q+d+1$ suffices that
the semidefinite program in \cite{CaFe14} indeed solves the total variation minimization problem for nonnegative measures in all dimensions $d$.
In particular, this gives a sharp constant in \cite[Thm.~1.2]{CaFe14} and bypasses the relaxation from a nonnegative trigonometric polynomial
to the sum of squares representation, known to possibly increase degrees for $d=2$ and to possibly fail for $d>2$, cf.~\cite[Remark~4.17, Theorem~4.24, and Remark~4.26]{Du07}.

\subsection{Spherical harmonics and parameters on the sphere}\label{subsect:sphere}
Now let $\K=\R$, restrict to parameters on the unit sphere $\S^{d-1}=\{x\in\R^d: x^\top x=1\}=\V(1-\sum_{j=1}^dX_j^2)$ in the $d$-dimensional Euclidean space, and we refer to \cite{Mue66,Sz75,AtHa12} for an introduction to approximation on the sphere and spherical harmonics.
The polynomials in $d$ variables of degree up to $n$ restricted to the sphere can be decomposed into mutually orthogonal spaces
\begin{equation*}
 \Pi_n/\I_n(\S^{d-1})=\bigoplus_{k=0}^n H_k^d
\end{equation*}
of real spherical harmonics of degree $k\in\NZ$ and we let $\{Y_k^\ell:\S^{d-1}\rightarrow\C:\ell=1,\hdots,\dim(H_k^d)\}$
denote an orthonormal basis for each $H_k^d$.
The dimension of these spaces obeys $N_k:=\dim(H_k^d)={\left(2k+d-2\right)\Gamma\left(k+d-2\right)}/({\Gamma\left(k+1\right)\Gamma\left(d-1\right)})$ for $k\ge 1$ and
we let $N:=\sum_{k=0}^n N_k=\mathcal{O}(n^{d-1})$ denote the dimension of $\Pi_n/\I_n(\S^{d-1})$.

Now let $M\in\N$, coefficients $\hat{f}_j\in\R\setminus\set{0}$, and pairwise distinct $x_j\in\S^{d-1}$, $j=1,\dots,M$, be given.
Then the moment sequence of the signed Dirac ensemble $\mu:\mathcal{P}(\S^{d-1})\rightarrow\R$, $\mu=\sum_{j=1}^M\hat{f}_j\delta_{x_j}$,
is the spherical harmonic sum
\begin{equation*}
f\colon\{(k,\ell):k\in\NZ, \ell=1,\hdots,N_k\}\to\R,
\quad
(k,\ell)\mapsto\int_{\S^{d-1}} Y_k^{\ell}(x) \dd\mu(x)
=\sum_{j=1}^M\hat{f}_j Y_k^{\ell}(x_j),
\end{equation*}
with parameters $x_j\in\S^{d-1}$.
Finally, we define the multivariate Vandermonde matrix a.k.a. nonequispaced spherical Fourier matrix
\begin{equation*}
 Y_n:=\left(Y_k^{\ell}(x_j)\right)_{\substack{j=1,\dots,M\\k\in\NZ, k\le n, \ell=1,\hdots,N_k}}\in\R^{M\times N}.
\end{equation*}

Regarding the reconstruction of the measure from its first moments, we have the following results.
\begin{lemma}\label{lem:VkerY}
Let $\Omega:=\set{x_j:j=1,\dots,M}\subset\S^{d-1}$ and $n\in\N_0$ such that $Y_n$ has full rank $M$, then
$\Omega=\V_{\S^{d-1}}(\ker Y_{n+1})$.
\end{lemma}
\begin{proof}
Note that $Y_n\in\R^{M\times N}$ is the matrix of the $\R$-linear map $\calA_{\S^{d-1},n}^\Omega\colon\Pi_{\S^{d-1},n}\to\R^\Omega\cong\R^M$
w.r.t.~the basis $\bigcup_{k=0}^n\set{Y^\ell_k:\ell=1,\dots,\dim(H^d_k)}$ of $\Pi_{\S^{d-1},n}$ and the canonical basis of $\R^M$.
Since $\rank Y_n=M$ by assumption, $\calA_{\S^{d-1},n}^\Omega$ is surjective
and the assertion is an immediate consequence of Corollary~\ref{cor:Moeller}.
\end{proof}

\begin{lemma}[\protect{\cite[Thm.~2.4]{Ku07}}]\label{lem:inghamSd}
 For $\Omega:=\set{x_j:j=1,\dots,M}\subset\S^{d-1}$ let 
  \begin{equation*}
   \sep(\Omega):=\min_{j\ne\ell} \arccos\left(x_j^\top x_{\ell}\right)
  \end{equation*}
 denote the separation distance and call the set of parameters $q$-separated if $\sep(\Omega)>q$.
 Now if $n\in\NZ$ fulfills $n>2.5\pi d/q$, then the matrix $Y_n\in\C^{M\times N}$ has full rank $M$.
\end{lemma}

\begin{thm}\label{thm:sphMain}
 Let $f\colon \{(k,\ell):k\in\NZ, \ell=1,\hdots,N_k\}\to\R$ be an $M$-sparse spherical harmonic sum with parameters $x_j\in\S^{d-1}$, $j=1,\dots,M$.
 If the parameters are $q$-separated and $n>2.5\pi d/q+1$, then
 \begin{equation*}
  \Omega=\V_{\S^{d-1}}(\ker \tilde H_n)
 \end{equation*}
 where the entries of the matrix
 \begin{equation*}
  \tilde H_n:=Y_n^{\top} D Y_n\in\R^{N\times N},\qquad D=\diag(\hat f_1,\hdots,\hat f_M),
 \end{equation*}
 mimicking \eqref{eq:factorH}, can be computed solely from the moments $f(k,\ell)$, $k\le 2n$, $\ell=1,\hdots,N_k$.
\end{thm}
\begin{proof}
 We just combine Lemmata \ref{lem:VkerY}, \ref{lem:inghamSd}, and proceed as in Theorem \ref{thm:trigMain} to show $\ker\tilde H_n=\ker Y_n$.
 Finally note that $Y_k^{\ell}\cdot Y_r^s =\sum_{t=0}^{k+s}\sum_{u=1}^{N_t} c_{k,r,t}^{\ell,s,u} Y_t^u$ with some
 Clebsch-Gordan coefficients and thus
 \begin{equation*}
  (\tilde H_n)_{(k,\ell),(r,s)}=\sum_{j=1}^M \hat f_j Y_k^{\ell}(x_j) Y_r^s(x_j)=\sum_{t=0}^{k+s}\sum_{u=1}^{N_t} c_{k,r,t}^{\ell,s,u} f(t,u).
 \end{equation*}
\end{proof}

 Finally note that the semidefinite program in \cite{BeDeFe15a,BeDeFe15b} indeed solves the total variation minimization problem for nonnegative measures on spheres in all dimensions $d$ provided the order of the moments is large enough as shown by the following construction of a dual certificate and sum of squares representation.
\begin{corollary}\label{cor:DualCertificate}
 Let $d,n,M\in\N$, $\Omega=\{x_j\in\S^{d-1}:j=1,\hdots,M\}$ be $q$-separated, and $n>2.5\pi d/q+1$.
 Moreover, let $\hat p_r\in\R^N$, $r=1,\hdots,N$, be an orthonormal basis with $\hat p_r\in\ker(Y_n)^{\bot}$, $r=1,\hdots,M$,
 and $p_r:\S^{d-1}\rightarrow\R$, $p_r=\sum_{k=0}^n \sum_{\ell=1}^{N_k} \hat p_{r,k}^{\ell} Y_k^\ell$, then $p:\S^{d-1}\rightarrow\R$,
 \begin{equation*}
  p(x)=\frac{2\pi^{d/2}}{\Gamma(d/2)N} \sum_{r=1}^M |p_r(x)|^2,
 \end{equation*}
 is a polynomial on the sphere of degree at most $2n$ and fulfills $0\le p(x) \le 1$ for all $x\in \S^d$ and $p(x)=1$ if and only if $x\in\Omega$.
\end{corollary}
\begin{proof}
 First note that every orthonormal basis $\hat p_{\ell}\in\R^N$, $\ell=1,\hdots,N$, leads to
 \begin{equation*}
  \sum_{r=1}^N |p_r(x)|^2
  = \sum_{k,u=0}^n\sum_{\ell,v=1}^{N_k} Y_k^{\ell}(x) Y_u^v(x) \sum_{r=1}^N  \hat p_{r,k}^{\ell} \hat p_{r,u}^{v}
  = \sum_{k=0}^n\sum_{\ell=1}^{N_k} Y_k^{\ell}(x) Y_k^{\ell}(x)=\frac{\Gamma(d/2)N}{2\pi^{d/2}}
 \end{equation*}
 for $x\in\S^{d-1}$, where the last equality is due to the addition theorem for spherical harmonics and as in the proof of Theorem~\ref{thm:sphMain}, the product
 $Y_k^{\ell} \cdot Y_k^{\ell}$ always is a polynomial on the sphere of degree at most $2k$.
 Finally, Theorem~\ref{thm:sphMain} assures $\sum_{r=M+1}^N |p_r(x)|^2=0$ if and only if $x\in\Omega$.
\end{proof}


 \begin{example}
  We conduct the following two small scale numerical examples. For $M=3$ points on the unit sphere and a polynomial degree $n=2$, we compute the $N-M=6$ dimensional kernel of the nonequispaced spherical Fourier matrix $Y_n$, set up the corresponding kernel polynomials $p_r$, $r=4,\hdots,9$, as defined in Corollary \ref{cor:DualCertificate} and plot the surface
  $q(x)=1+\frac{1}{2}\min_{r=M+1,\hdots,N} |p_r(x)|^{1/4}$, $x\in\S^2$, in Figure \ref{fig:d=1}(a). The absolute value of each kernel polynomial forms a valley around its zero set which gets narrowed by the $4$-th root and the minimum over all these valleys visualizes the common zeros as junction points in this surface.
  
  In a second experiment, we consider $M=50$ random points $x_j$ on the unit sphere, an associated Dirac ensemble with random coefficients $\hat f_j$, and its moments up to order $60$, i.e., $n=30$.
  The $M$-dimensional orthogonal complement of the kernel of the matrix $\tilde H_n$ defines the so-called signal space.
  Figure \ref{fig:d=1}(b) clearly shows that the dual certificate $p$, defined as in Corollary \ref{cor:DualCertificate}, peaks exactly at the points $x_j$.
 \begin{figure}[htbp]
\centering
  \subfigure[Visualization of the polynomials in the kernel of $Y_n$, $M=3$ points, $n=2$, plot of the surface $q(x)=1+\frac{1}{2}\min_{r=M+1,\hdots,N} |p_r(x)|^{1/4}$, $x\in\S^2$.]
  {\includegraphics[width=0.49\textwidth]{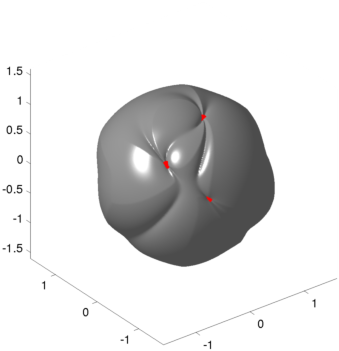}}
  \subfigure[$M=50$ random points on $\S^2$ are identified from the moments of order $\le 60$. The dual certificate $p$ is plotted as surface $1+\frac{1}{2}p(x)$, $x\in\S^2$.]
  {\includegraphics[width=0.49\textwidth]{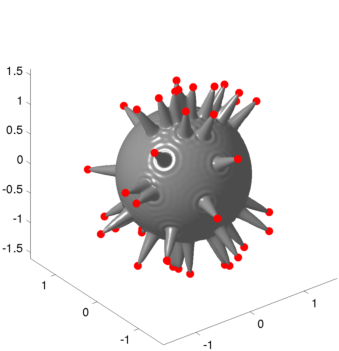}}
  \caption{Visualization of kernel polynomials and dual certificate on the sphere $\S^2$.}
\label{fig:d=1}
\end{figure}
\end{example}

\section{Summary}\label{sect:sum}

We considered a recently developed multivariate generalization of Prony's method, characterized its succeeding in terms of an interpolation condition, and gave a generalization to the sphere.
The interpolation condition is shown to hold for separated points in the trigonometric and the spherical case in arbitrary dimensions and also yield a certificate for popular semidefinite relaxations of the reconstruction problems.
Beyond the scope of this paper, future research needs to address the actual computation of the points and the stability under noise.

\textbf{Acknowledgment.}
 We gratefully acknowledge support
 by the DFG within the research training group 1916: Combinatorial structures in geometry
 and by the Helmholtz Association within the young investigator group VH-NG-526: Fast algorithms for biomedical imaging.

\bibliographystyle{abbrv}
\bibliography{../references}
\end{document}